\documentclass[12pt]{amsart}
\usepackage{amssymb}
\usepackage{hyperref}
\usepackage{tabularx}
\usepackage{booktabs}
\usepackage{caption}

\newtheorem{thm}{Theorem}[section]
\newtheorem{lem}[thm]{Lemma}

\newtheorem{prop}[thm]{Proposition}
\newtheorem{ex}[thm]{Example}

\newtheorem*{prob*}{Open problem}

\theoremstyle{definition}

\newtheorem{defi}[thm]{Definition}

\theoremstyle{remark}

\newtheorem{rem}[thm]{Remark}
\newtheorem*{rem*}{Remark}

\newcommand{\kringel}{\mathbin{\raise1pt\hbox{$\scriptstyle\circ$}}}
\newcommand{\pkt}{\mathbin{\raise0pt\hbox{$\scriptstyle\bullet$}}}

\newcommand{\C}{\mathbb{C}}

\newcommand{\ad}{{\rm ad}}

\newcommand{\End}{{\rm End}}
\newcommand{\Der}{{\rm Der}}

\newcommand{\La}{\mathfrak{a}}
\newcommand{\Lb}{\mathfrak{b}}

\newcommand{\Lg}{\mathfrak{g}}
\newcommand{\Lh}{\mathfrak{h}}

\newcommand{\Ll}{\mathfrak{l}}
\newcommand{\Ln}{\mathfrak{n}}

\newcommand{\Ls}{\mathfrak{s}}
\newcommand{\Lt}{\mathfrak{t}}

\newcommand{\al}{\alpha}
\newcommand{\be}{\beta}
\newcommand{\ga}{\gamma}
\newcommand{\de}{\delta}
\newcommand{\ep}{\varepsilon}

\newcommand{\la}{\lambda}

\newcommand{\ra}{\rightarrow}

\renewcommand{\phi}{\varphi}

\begin{document}


\title[CPA-structures]{Commutative Post-Lie algebra structures on nilpotent Lie algebras and Poisson algebras}

\author[D. Burde]{Dietrich Burde}
\author[C. Ender]{Christof Ender}
\address{Fakult\"at f\"ur Mathematik\\
Universit\"at Wien\\
  Oskar-Morgenstern-Platz 1\\
  1090 Wien \\
  Austria}
\email{dietrich.burde@univie.ac.at}
\address{Fakult\"at f\"ur Mathematik\\
Universit\"at Wien\\
  Oskar-Morgenstern-Platz 1\\
  1090 Wien \\
  Austria}
\email{christof.ender@univie.ac.at}

\date{\today}


\begin{abstract}
We give an explicit description of commutative post-Lie algebra structures on some classes of nilpotent Lie algebras. 
For non-metabelian filiform nilpotent Lie algebras and Lie algebras of strictly upper-triangular matrices we show that 
all CPA-structures are associative and induce an associated Poisson-admissible algebra.
 \end{abstract}

\maketitle

\section{Introduction}

Post-Lie algebras and post-Lie algebra structures arise in many different contexts. We have introduced these structures in
\cite{BU41} to characterize simply transitive nil-affine actions of a nilpotent Lie group $G$ on another
nilpotent Lie group $N$. This plays an important role in the theory of nil-affine crystallographic groups and 
complete nil-affinely flat manifolds. Post-Lie algebras arise there as a natural common generalization of 
pre-Lie algebras \cite{HEL,KIM,SEG,BU5,BU19,BU24} and LR-algebras \cite{BU34, BU38}, and the geometric questions can be 
formulated on the level of post-Lie algebras. \\
On the other hand, post Lie algebras have been introduced independently by Vallette \cite{VAL} in connection with the 
homology of partition posets and the study of Koszul operads. Since then several authors have studied these algebras in 
various other contexts, e.g., for algebraic operad triples \cite{LOD},
in connection with modified Yang-Baxter equations, Rota-Baxter operators, universal enveloping algebras, 
double Lie algebras, classical $R$-matrices, isospectral flows, Lie-Butcher series and many other topics 
\cite{BAI, BU57, BU59, ELM}. \\[0.2cm]
An important question arising from geometry here is the {\em existence question of post-Lie algebra structures} for 
a given pair of Lie algebras. This is in general a very hard question and the answer depends very much on the algebraic 
properties of the two given Lie algebras. For a survey on the results and open problems see \cite{BU33,BU41,BU44}. 
An important class of post-Lie algebra structures is given by {\em commutative} structures, so-called {\em CPA-structures}. 
These structures are much more accessible than general post-Lie algebra structures and we can answer
several questions concerning existence and classification, thereby finding directions to understand the general 
case of post-Lie algebra structures. For CPA-structures on semisimple, perfect and complete Lie algebras, see \cite{BU51,BU52}.
For CPA-structures on nilpotent Lie algebras, see  \cite{BU57}. \\
In this paper we also show that CPA-structures are related to other algebraic structures coming from geometry, namely to 
Poisson structures \cite{BEB} and Poisson algebras \cite{GOR}. Indeed, for many classes of Lie algebras, CPA-structures 
are associative and induce a Poisson admissible algebra structure. \\[0.2cm]
The paper is structured as follows. In section $2$ we show that for a CPA structure $(V,\cdot)$ on $\Lg$ we have 
$\Lg\cdot [\Lg,\Lg]=0$ if and only if $(V,\cdot)$ is an associative algebra, if and only if it is a Poisson-admissible algebra.
We prove that $\Lg\cdot \Lg\subseteq Z(\Lg)$ implies $\Lg\cdot [\Lg,\Lg]=0$ and that $(V,\cdot, [,\, ])$ is a Poisson
algebra if and only if the CPA-structure is central, i.e., satisfies $\Lg\cdot \Lg\subseteq Z(\Lg)$. \\[0.2cm]
In section $3$ we show that every CPA-structure $(V,\cdot)$ on a non-metabelian complex filiform Lie algebra
is associative, i.e., satisfies  $\Lg\cdot [\Lg,\Lg]=0$ so that $(V,\cdot)$ is Poisson-admissible. The result does
not hold for metabelian filiform Lie algebras, where we only obtain $[\Lg,\Lg]\cdot [\Lg,\Lg]=0$. For special families of
filiform Lie algebras we can give an explicit description of all CPA-structures. \\[0.2cm]
In section $4$ we show that every CPA-structure $(V,\cdot)$ on the nilpotent Lie algebra $\Ln_n(K)$, $n\ge 5$ of 
$n\times n$ strictly upper-triangular matrices over $K$ is associative and $(V,\cdot)$ is Poisson-admissible. \\
This paper is based on results of the PhD thesis \cite{END} of the second author, where further details can be found.

\section{Preliminaries}

Let $K$ denote a field of characteristic zero. We have defined a post-Lie algebra structure 
on a pair of Lie algebras $(\Lg,\Ln)$ over $K$ in \cite{BU41} as follows:

\begin{defi}\label{pls}
Let $\Lg=(V, [\, ,])$ and $\Ln=(V, \{\, ,\})$ be two Lie brackets on a vector space $V$ over 
$K$. A {\it post-Lie algebra structure}, or {\em PA-structure} on the pair $(\Lg,\Ln)$ is a 
$K$-bilinear product $x\cdot y$ satisfying the identities
\begin{align}
x\cdot y -y\cdot x & = [x,y]-\{x,y\} \label{post1}\\
[x,y]\cdot z & = x\cdot (y\cdot z) -y\cdot (x\cdot z) \label{post2}\\
x\cdot \{y,z\} & = \{x\cdot y,z\}+\{y,x\cdot z\} \label{post3}
\end{align}
for all $x,y,z \in V$.
\end{defi}

Define by  $L(x)(y)=x\cdot y$ and $R(x)(y)=y\cdot x$ the left respectively right multiplication 
operators of the algebra $A=(V,\cdot)$. By \eqref{post3}, all $L(x)$ are derivations of the Lie 
algebra $(V,\{,\})$. Moreover, by \eqref{post2}, the left multiplication
\[
L\colon \Lg\ra \Der(\Ln)\subseteq \End (V),\; x\mapsto L(x)
\]
is a linear representation of $\Lg$. The right multiplication $R\colon V\ra V,\; x\mapsto R(x)$
is a linear map, but in general not a Lie algebra representation. \\
If $\Ln$ is abelian, then a post-Lie algebra structure on $(\Lg,\Ln)$ corresponds to
a {\it pre-Lie algebra structure} on $\Lg$. In other words, if $\{x,y\}=0$ for all $x,y\in V$, then 
the conditions reduce to
\begin{align*}
x\cdot y-y\cdot x & = [x,y], \\
[x,y]\cdot z & = x\cdot (y\cdot z)-y\cdot (x\cdot z),
\end{align*}
i.e., $x\cdot y$ is a {\it pre-Lie algebra structure} on the Lie algebra $\Lg$. 
If $\Lg$ is abelian, then the conditions reduce to
\begin{align*}
x\cdot y-y\cdot x & = -\{x,y\} \\
x\cdot (y\cdot z)& = y\cdot (x\cdot z), \\
x\cdot \{y,z\} & = \{x\cdot y,z\}+\{y,x\cdot z\},
\end{align*}
i.e., $-x\cdot y$ is an {\it LR-structure} on the Lie algebra $\Ln$. For details see \cite{BU41}. \\[0.2cm]
An important case of a post-Lie algebra structure arises if the algebra $A=(V,\cdot)$ is {\it commutative}, i.e., 
if $x\cdot y=y\cdot x$ is satisfied for all $x,y\in V$, so that we have $L(x)=R(x)$ for all $x\in V$. Then the two 
Lie brackets $[x,y]=\{x,y\}$ coincide, and we obtain a commutative algebra structure on $V$ associated with only one Lie 
algebra \cite{BU51}:

\begin{defi}\label{cpa}
A {\it commutative post-Lie algebra structure}, or {\em CPA-structure} on a Lie algebra $\Lg$ 
is a $K$-bilinear product $x\cdot y$ satisfying the identities:
\begin{align}
x\cdot y & =y\cdot x \label{com4}\\
[x,y]\cdot z & = x\cdot (y\cdot z) -y\cdot (x\cdot z)\label{com5} \\
x\cdot [y,z] & = [x\cdot y,z]+[y,x\cdot z] \label{com6}
\end{align}
for all $x,y,z \in V$. 
\end{defi}

We often write $(V,\cdot)$ for a CPA-structure on $\Lg$. \\[0.2cm]
It turns out that certain CPA-structures are related to Poisson algebras and Poisson-admissible algebras. Such algebras
also have been studied in a geometric and algebraic context, see for example \cite{BEB, GOR}. The definition is as follows.

\begin{defi}
A {\em Poisson algebra} is a triple $(V,\circ, [\, ,\,] )$, where $(V,\circ)$ is a commutative, associative
algebra, and $\Lg=(V,[\, ,\,])$ is a Lie algebra such that the identity
\begin{align}
[x, y\circ z] & = [z,x]\circ y + x\circ [z,y] \label{poi}
\end{align}
holds for all $x,y,z\in V$. 
\end{defi}

Recall that a non-associative algebra is a vector space $V$ together with a bilinear product $V\times V\ra V$, 
$(x,y)\mapsto x\cdot y$. The product need not be associative in general.

\begin{defi}
A non-associative algebra $(V,\cdot)$ is called {\em Poisson-admissible} if 
$(V,\circ, [\, ,\,] )$, given by
\begin{align*}
[x,y] & = x\cdot y-y\cdot x,\\
x\circ y & = \frac{1}{2}( x\cdot y+y\cdot x)
\end{align*}
is a Poisson algebra.
\end{defi}

A CPA-structure $(V,\cdot)$ on $\Lg$ often satisfies $\Lg\cdot [\Lg,\Lg]=0$. This means that $(V,\cdot)$ is a Poisson-admissible
algebra.

\begin{lem}\label{2.5}
Let $(V,\cdot)$ be a CPA-structure on a Lie algebra $\Lg$. Then the following properties are equivalent.
\begin{itemize}
\item[$(1)$] $\Lg\cdot [\Lg,\Lg]=0$. 
\item[$(2)$] $(V,\cdot)$ is an associative algebra. 
\item[$(3)$] $(V,\cdot)$ is a Poisson-admissible algebra.
\end{itemize}
\end{lem}

\begin{proof}
Using \eqref{com4} and  \eqref{com5} we have
\begin{align*}
(x\cdot y)\cdot z-x\cdot (y\cdot z) & = z\cdot (x\cdot y)-x\cdot (y\cdot z) \\
                                    & = [z,x]\cdot y+x\cdot(z\cdot y)-x\cdot (y\cdot z) \\
                                    & = y\cdot [z,x]
\end{align*}
for all $x,y,z\in V$. This shows that $(1)\Leftrightarrow (2)$. Since $(V,\cdot)$ is commutative we have
\begin{align*}
x\circ y & =\frac{1}{2}(x\cdot y+y\cdot x)=x\cdot y, \\
[x,y] & = x\cdot y-y\cdot x =0,
\end{align*}
and \eqref{poi} is trivially satisfied. Hence $(V,\circ, [\, ,\,] )$ is a Poisson algebra if and only if $(V,\cdot)$ is associative. 
This shows that $(2)\Leftrightarrow (3)$.
\end{proof}

The lemma motivates the following definition.

\begin{defi}
A CPA-structure $(V,\cdot)$ on $\Lg$ is called {\em associative} if $\Lg\cdot [\Lg,\Lg]=0$. It is called
{\em central} if $\Lg\cdot \Lg\subseteq Z(\Lg)$. 
\end{defi}

We have the following implications.

\begin{lem}\label{2.7}
Every central CPA-structure on $\Lg$ is associative and every associative CPA-structure on $\Lg$ satisfies 
$\Lg\cdot \Lg\subseteq Z([\Lg,\Lg])$.
\end{lem}

\begin{proof}
Assume that $\Lg\cdot \Lg\subseteq Z(\Lg)$. Then by \eqref{com6} we have
\[
x\cdot [y,z]=[x\cdot y,z]+[y,x\cdot z]=0
\]
for all $x,y,z\in V$. Hence we obtain that $\Lg\cdot [\Lg,\Lg]=0$. Conversely suppose that
$\Lg\cdot [\Lg,\Lg]=0$ and let $x,y\in \Lg$ and $z\in [\Lg,\Lg]$. Then $x\cdot z=0$ by assumption, so that by \eqref{com6}
we have
\[
[x\cdot y,z] = x\cdot [y,z]-[y, x\cdot z] = 0.
\]
Hence we have $\Lg\cdot \Lg \subseteq Z([\Lg,\Lg])$.
\end{proof}

We have studied central CPA-structures on nilpotent Lie algebras in \cite{BU57} and shown the following result, see Theorem $4.3$.

\begin{thm}
All CPA-structures on a free-nilpotent Lie algebra $F_{3,c}$ with $3$ generators and nilpotency class $c\ge 3$ are central.
\end{thm}

This result should also hold for every number of generators $g\ge 3$ and nilpotency class $c\ge 3$. 
 
\begin{lem}
Let $(V,\cdot)$ be a CPA-structure on $\Lg$. Then $(V,\cdot, [\, ,\,] )$ is a Poisson algebra if and only if $(V,\cdot)$ is central. 
\end{lem}

\begin{proof}
Suppose that $(V,\cdot)$ is central. Then it is also associative by Lemma $\ref{2.7}$. 
Hence  $(V,\cdot)$ is commutative and associative and \eqref{poi} for the product $x\cdot y$ becomes
\[
[x,y\cdot z]=[x,y]\cdot z+y\cdot [x,z],
\]
which is satisfied, because every term is zero. Hence $(V,\cdot, [\, ,\,] )$ is a Poisson algebra. 
Conversely, if $(V,\cdot, [\, ,\,] )$ 
is a Poisson algebra, then $(V,\cdot)$ is associative, hence satisfies $\Lg\cdot [\Lg,\Lg]=0$ by Lemma $\ref{2.5}$. 
Hence we have $[x,y\cdot z]=[x,y]\cdot z+y\cdot [x,z]=0$ for all $x,y,z\in V$ so that $\Lg\cdot \Lg\subseteq Z(\Lg)$.
\end{proof}

Also certain LR-structures are related to Poisson-admissible algebras.

\begin{lem}\label{2.10}
Let $(V,\cdot)$ be an LR-structure on a Lie algebra $(\Ln,\{,\, \})$. Then the following properties are equivalent.
\begin{itemize}
\item[$(1)$] $\Ln\cdot \Ln\subseteq Z(\Ln)$. 
\item[$(2)$] $(V,\cdot)$ is an associative algebra. 
\item[$(3)$] $(V,\cdot)$ is a Poisson-admissible algebra.
\end{itemize}
\end{lem}

\begin{proof}
Since $u\cdot v-v\cdot u=-\{u,v\}$ we have $z\cdot (x\cdot y)-(x\cdot y)\cdot z=-\{z,x\cdot y\}$ and also
$x\cdot (z\cdot y)-x\cdot (y\cdot z)=-x\cdot \{z,y\}$. Hence the associator is given by
\begin{align*}
x\cdot (y\cdot z)-(x\cdot y)\cdot z & = x\cdot (y\cdot z)-z\cdot (x\cdot y)-\{z,x\cdot y\} \\
                                    & = x\cdot (y\cdot z)-x\cdot (z\cdot y)-\{z,x\cdot y\} \\
                                    & = x\cdot \{z,y\}-\{z,x\cdot y\} \\
                                    & = \{x\cdot z,y\},
\end{align*}
which vanishes if and only if $\Ln\cdot \Ln\subseteq Z(\Ln)$. This shows that $(1)\Leftrightarrow (2)$.
For the equivalence $(2)\Leftrightarrow (3)$, see Proposition $3.2$ in \cite{BEB}.
\end{proof}

Again we call an LR-structure on $\Ln$ {\em associative}, if one of the above conditions is satisfied.
We obtain the following interesting consequence.

\begin{prop}
Let $(V,\cdot)$ be an associative LR-structure on a Lie algebra $\Ln$. Then $\Ln$ is $2$-step nilpotent.
\end{prop}

\begin{proof}
By Lemma $\ref{2.10}$  $(V,\cdot)$ is associative if and only it is Poisson-admissible. In this case
the Lie algebra $\Ln$ is $2$-step nilpotent by Corollary $3.1$ in \cite{BEB}. 
\end{proof}

\begin{prop}\label{2.12}
Let $(V,\cdot)$ be an associative LR-structure on a Lie algebra $\Ln$ and suppose that $Z(\Ln)\subseteq \{\Ln,\Ln\}$. 
Then we have
\[
\Ln\cdot (\Ln\cdot \Ln)=(\Ln\cdot \Ln)\cdot \Ln=0.
\]
In particular, all left multiplications $L(x)$ are nilpotent.
\end{prop}

\begin{proof}
Because of $\Ln\cdot \Ln\subseteq Z(\Ln)$ and $Z(\Ln)\subseteq \{\Ln,\Ln\}$ we have
\begin{align*}
\Ln\cdot (\Ln\cdot \Ln) & \subseteq \Ln\cdot Z(\Ln) \subseteq \Ln\cdot \{\Ln,\Ln\} \\
                        & \subseteq \{\Ln\cdot \Ln,\Ln\} + \{\Ln,\Ln\cdot \Ln \} \\
                        & \subseteq \{Z(\Ln),\Ln\}=0.
\end{align*}
Similarly we have $(\Ln\cdot \Ln)\cdot \Ln=0$.
\end{proof}

\begin{rem}
Proposition $\ref{2.12}$ is no longer true for LR-structures which are not associative. The classification of
LR-structures on the $3$-dimensional Heisenberg Lie algebra $\Ln=\Ln_3(K)$, see Proposition $3.1$ of \cite{BU34},
gives a counterexample. Let $e_1,e_2,e_3$ be a basis of $V$ with $\{e_1,e_2\}=e_3$. The LR-structure $A_4$, given 
by the products
\begin{align*}
e_2\cdot e_1 & = -e_3,\quad  e_2\cdot e_3=e_3,\\
e_2\cdot e_2 & = e_2,\hspace{0.76cm} e_3\cdot e_2=e_3
\end{align*}
is not associative, since $e_2\cdot (e_1\cdot e_2)-(e_2\cdot e_1)\cdot e_2=e_3$, and $L(e_2)$ is not nilpotent. 
In particular, $\Ln\cdot (\Ln\cdot \Ln)\neq 0$.
\end{rem}

\section{CPA-structures on filiform Lie algebras}

Let $\Lg$ be a Lie algebra. The {\it lower central series} of $\Lg$ is given by
\[
\Lg^0=\Lg \supseteq \Lg^1 \supseteq \Lg^2 \supseteq \Lg^3  \supseteq \cdots
\]
where the ideals $\Lg^i$ are defined by $\Lg^i=[\Lg,\Lg^{i-1}]$ for all $i\ge 1$. The {\it derived series} of $\Lg$ is given by
\[
\Lg^{(0)}=\Lg \supseteq \Lg^{(1)} \supseteq \Lg^{(2)} \supseteq \Lg^{(3)}  \supseteq \cdots
\]
where the ideals $\Lg^{(i)}$ are defined by $\Lg^{(i)}=[\Lg^{(i-1)},\Lg^{(i-1)}]$
for all $i\ge 1$. 
The {\em nilpotency class} $c(\Lg)$ is the smallest integer $k\ge 1$ with $\Lg^k=0$ and the {\em solvability class}
$d(\Lg)$ is the smallest integer $k\ge 1$ with $\Lg^{(k)}=0$.

\begin{defi}
A Lie algebra $\Lg$ is called {\em metabelian} if it is solvable of class $d(\Lg)\le 2$. It is called {\em filiform} 
if it is nilpotent of class $c(\Lg)=\dim(\Lg)-1$. Furthermore $\Lg$ is called a {\em stem} Lie algebra if
$Z(\Lg)\subseteq [\Lg,\Lg]$.
\end{defi}

Let $\Lg$ be a complex filiform Lie algebra of dimension $n$. Then there exists an {\em adapted basis} $(e_1,\ldots ,e_n)$
of $\Lg$, which among other relations satisfies $[e_1,e_i]=e_{i+1}$ for $2\le i\le n-1$. For $j\ge 1$ define the characteristic 
ideals
\[
I_j={\rm span}\{e_j,\ldots ,e_n\},
\]
which refine the lower central series of $\Lg$ with $I_1=\Lg$ and $I_j=\Lg^{j-2}$ for $j\ge 3$. Note that
$I_n=Z(\Lg)\subseteq [\Lg,\Lg]=I_3$ and $I_{n+i}=0$ for all $i\ge 1$.

\begin{lem}\label{3.2}
Let $(V,\cdot)$ be a CPA-structure on a filiform Lie algebra $\Lg$ and $(e_1,\ldots ,e_n)$ be an adapted basis of
$\Lg$. Then $L(e_i)(I_j)\subseteq I_{j+1}$ for $1\le i,j\le n$. In particular we have $\Lg\cdot \Lg\subseteq I_2$ and
$\Lg\cdot I_2\subseteq I_3$. Furthermore it holds
\begin{align*}
\Lg\cdot \Lg \subseteq I_3 & \Longleftrightarrow e_1\cdot e_1\in I_3,\\
\Lg\cdot I_2 \subseteq I_4 & \Longleftrightarrow e_1\cdot e_2\in I_4, \; e_2\cdot e_2\in I_4.
\end{align*}
\end{lem}

\begin{proof}
Since the ideals $I_j$ are characteristic and the left multiplications $L(e_i)$ are derivations, it follows that
$L(e_i)(I_j)\subseteq I_j$ for all $1\le i,j\le n$. Since filiform Lie algebras are nilpotent stem Lie algebras, all left
multiplications are nilpotent by Theorem $3.6$ in \cite{BU57}. Hence we have $L(e_i)(I_j)\subseteq I_{j+1}$ for $1\le i,j\le n$.
This implies that $\Lg\cdot I_2\subseteq I_3$, so that $\Lg\cdot \Lg\subseteq I_3$ follows from $e_1\cdot e_1\in I_3$.
We also have $\Lg\cdot I_3\subseteq I_4$, so that $\Lg\cdot I_2\subseteq I_4$ follows from $e_1\cdot e_2,\, e_2\cdot e_2\in I_4$.
\end{proof}

\begin{lem}\label{3.3}
Let $(V,\cdot)$ be a CPA-structure on a filiform Lie algebra $\Lg$ and $(e_1,\ldots ,e_n)$ be an adapted basis of
$\Lg$ with $\Lg\cdot \Lg \subseteq I_3$ and $\Lg\cdot I_2 \subseteq I_4$. Suppose that for some $\ell\ge 0$ we have
\begin{align*}
e_1\cdot e_j & \in I_{j+\ell+2} \text{ for all } 3\le j \le n, \\
e_i\cdot e_j & \in I_{i+j+\ell} \text{ for all } (i,j)\neq (1,1),(1,2),(2,1),(2,2).
\end{align*} 
Then the same is true for $\ell+1$, i.e., we have
\begin{align*}
e_1\cdot e_j & \in I_{j+\ell+3} \text{ for all } 3\le j \le n, \\
e_i\cdot e_j & \in I_{i+j+\ell+1} \text{ for all } (i,j)\neq (1,1),(1,2),(2,1),(2,2).
\end{align*} 
\end{lem}

\begin{proof}
We first show by induction on $j\ge 3$ that $e_2\cdot e_j\in I_{3+j+\ell}$. For $j=3$ we have using \eqref{com4} and
\eqref{com5}
\begin{align*}
e_2\cdot e_3 & = e_3\cdot e_2 \\
             & = [e_1,e_2]\cdot e_2 \\
            & = e_1\cdot (e_2\cdot e_2)-e_2\cdot (e_1\cdot e_2).
\end{align*}
Since $e_2\cdot e_2\in I_4$, $e_1\cdot e_2 \in I_4$ we have $e_1\cdot (e_2\cdot e_2) \in e_1\cdot I_4 \subseteq I_{4+\ell+2}=I_{6+\ell}$ 
by assumption and also $e_2\cdot (e_1\cdot e_2)\in e_2\cdot I_4\subseteq I_{6+\ell}$. It follows that $e_2\cdot e_3\in I_{6+\ell}$.
For the induction step $j\mapsto j+1$ we have
\begin{align*}
e_2\cdot e_{j+1} & = e_{j+1}\cdot e_2 \\
               & = [e_1,e_j]\cdot e_2 \\
               & = e_1\cdot (e_j\cdot e_2)-e_j\cdot (e_1\cdot e_2).
\end{align*}
By induction hypothesis we have $e_j\cdot e_2\in I_{3+j+\ell}$, so that $e_1\cdot (e_j\cdot e_2)\in I_{3+j+\ell+(\ell+2)}\subseteq
I_{4+j+\ell}$. Also $e_1\cdot e_2\in I_4$ implies that $e_j\cdot (e_1\cdot e_2)\in I_{j+4+\ell}$. It follows that 
$e_2\cdot e_{j+1}\in I_{4+j+\ell}$ and we have shown that $e_2\cdot e_j\in I_{3+j+\ell}$ for all $j\ge 3$. We can replace $e_2$
by $e_3,\ldots ,e_n$ and use induction as above. Then we obtain $e_i\cdot e_j\in I_{1+i+j+\ell}$ for all pairs $(i,j)$ which are not 
of the form $((1,j),(i,1),(2,2)$. To complete this for the remaining pairs $(i,j)$ it is enough to show the first 
claim, i.e., that $e_1\cdot e_j\in I_{j+\ell+3}$ for all $j\ge 3$. For $j=3$ we see that $e_1\cdot e_3\in I_{6+\ell}$ as above
and for the induction step  $j\mapsto j+1$ we have
\begin{align*}
e_1\cdot e_{j+1} & = e_{j+1}\cdot e_1 \\
             & = [e_1,e_j]\cdot e_1 \\
            & = e_1\cdot (e_j\cdot e_1)-e_j\cdot (e_1\cdot e_1).
\end{align*}
By induction hypothesis, $e_j\cdot e_1\in I_{j+\ell+3}$, so that $e_1\cdot (e_j\cdot e_1)\in I_{j+2\ell+5}\subseteq I_{4+j+\ell}$.
Furthermore $e_1\cdot e_1\in I_3$. Because of $e_j\cdot e_i\in I_{i+j+\ell+1}$ for $i,j\ge 3$, which we have shown before, it
follows that $e_j\cdot (e_1\cdot e_1)\in I_{3+j+\ell+1}=I_{4+j+\ell}$, and we are done. 
\end{proof}

We can state now our main result for filiform Lie algebras.

\begin{thm}\label{3.4}
Let $\Lg$ be a complex filiform Lie algebra of solvability class $d(\Lg)\ge 3$. Then every CPA-structure $(V,\cdot)$ on
$\Lg$ is associative and the algebra $(V,\cdot)$ is Poisson-admissible.
\end{thm}

\begin{proof}
Let $\dim(\Lg)=n$. Since every filiform Lie algebra of dimension $n\le 5$ is metabelian, we have $n\ge 6$ 
because of $d(\Lg)\ge 3$. For $n=6$ we can prove the result by a direct computation.
Every $6$-dimensional filiform Lie algebra $\Lg$ has an adapted basis $(e_1,\ldots, e_6)$ such
that the Lie brackets are given by
\begin{align*}
[e_1,e_i] & = e_{i+1},\; 2\le i\le 5, \quad  [e_2,e_5]=-\al_3 e_6, \\
[e_2,e_3] & = \al_1e_5+\al_2e_6, \hspace{1.0cm} [e_3,e_4] =\al_3 e_6,\\
[e_2,e_4] & = \al_1 e_6,
\end{align*}
for some $\al_1,\al_2,\al_3 \in \C$. Here $\Lg$ is metabelian if and only if $\al_3=0$. So we have $\al_3\neq 0$ and 
$d(\Lg)=3$. Let $(V,\cdot)$ be a CPA-structure on $\Lg$. Denote the
matrices for the left multiplications by 
\[
L(e_k)=(\zeta_{ij}^k), \; 1\le i,j\le n.
\]
All matrices are lower-triangular with respect to the basis $(e_1,\ldots, e_6)$ because of Lemma $\ref{3.2}$.
The identities \eqref{com4},\eqref{com5},\eqref{com6} are equivalent to a system of equations in the variables $\zeta_{ij}^k$, 
which can be easily solved directly. Indeed, \eqref{com4} and \eqref{com6} yield linear equations, which enables one to solve
the quadratic equations coming from \eqref{com5}. We obtain the following CPA-structures:
\begin{align*}
e_1\cdot e_1 & = \zeta_{51}^1e_5+\zeta_{61}^1e_6,\\
e_1\cdot e_2 & = -\al_3\zeta_{51}^1e_5+\zeta_{62}^1e_6,\\
e_2\cdot e_2 & = \al_3^2\zeta_{51}^1e_5+\zeta_{62}^2e_6.
\end{align*}
Recall that $e_1\cdot e_2=e_2\cdot e_1$. This shows that $\Lg\cdot [\Lg,\Lg]=0$, so that $(V,\cdot)$ is associative and a 
Poisson-admissible algebra by Lemma $\ref{2.5}$. \\[0.2cm]
We may now assume that $(V,\cdot)$ is a CPA-structure on $\Lg$ with $n\ge 7$. Denote again by $L(e_k)=(\zeta_{ij}^k)$ the 
left multiplications. We distinguish two cases, namely whether or not $\Lg\cdot \Lg\subseteq I_3$. This does not depend 
on the choice of an adapted basis for $\Lg$. The first case is again divided into two subcases whether or not 
$\Lg\cdot I_2\subseteq I_4$. \\[0.2cm]
{\em Case 1a:} It holds $\Lg\cdot \Lg\subseteq I_3$ and $\Lg\cdot I_2\subseteq I_4$. An adapted basis has the property that
$[e_i,e_j]\in I_{i+j}$ holds for all $1\le i,j\le n$ except for the case where $n$ is even, where we have the following
exceptional Lie brackets
\[
[e_i,e_{n-i+1}]=(-1)^{i+1}\al e_n, \; 2\le i\le n-1.
\]
Here the scalar $\al\in \C$ is zero if and only if $\Lg^{\lceil \frac{n-4}{2}\rceil}$ is abelian. \\[0.2cm]
{\em Case 1a1:} $\Lg^{\lceil \frac{n-4}{2}\rceil}$ is abelian. In this case we have $[I_i,I_j]\subseteq I_{i+j}$ for
all $1\le i,j\le  n$ with the convention that $I_m=0$ for all $m>n$. Using \eqref{com6} we obtain
\begin{align*}
e_1\cdot e_{i+1} & = e_1\cdot [e_1,e_i] \\
               & =[e_1\cdot e_1,e_i]+[e_1,e_1\cdot e_i]
\end{align*}
for all $2\le i\le n-2$. We have $e_1\cdot e_{i+1}=\zeta_{i+2,i+1}^1e_{i+2}+z$ with $z\in I_{i+3}$. On the other hand it follows from
$e_1\cdot e_1\in I_3$ that $[e_1\cdot e_1,e_i]\in I_{i+3}$ for all $2\le i\le n-2$. The case $i=n-2$ requires the assumption
that $\Lg^{\lceil \frac{n-4}{2}\rceil}$ is abelian, because we need that $[I_3,I_{n-2}]\subseteq I_{n+1}=0$.
Furthermore we have $[e_1,e_1\cdot e_i] =\zeta_{i+1,i}^1e_{i+2}+w$ with $w\in I_3$, so that
\[
\zeta_{i+2,i+1}^1e_{i+2}+z=\zeta_{i+1,i}^1e_{i+2}+w
\]
for all $2\le i\le n-2$. This implies that $\zeta_{j+1,j}^1=\zeta_{3,2}^1$ for all $2\le j\le n-2$.
Since $e_1\cdot e_2\in I_4$ we have $\zeta_{3,2}^1=0$, so that 
\[
\zeta_{j+1,j}^1=0 \text{ for all } 2\le j\le n-2.
\]
Hence the first subdiagonal of $L(e_1)$ is equal to zero. The same argument for $L(e_2)$ gives that
$\zeta_{j+1,j}^2=\zeta_{3,2}^2$ for all $2\le j\le n-2$. However, $\zeta_{3,2}^2=0$ because we have
\begin{align*}
e_2\cdot e_3 & = [e_1,e_2]\cdot e_2 \\
            & = e_1\cdot (e_2\cdot e_2) -e_2\cdot (e_1\cdot e_2).
\end{align*}
Indeed, $e_2\cdot e_3=\zeta_{3,2}^2e_4+z$ with $z\in I_5$, and $e_1\cdot (e_2\cdot e_2), e_2\cdot (e_1\cdot e_2) \in I_5$. So we 
have shown so far that $e_1\cdot e_i  \in I_{i+2}$ and $e_2\cdot e_i  \in I_{i+2}$ for all $1\le i\le n$.
It follows by induction on $i\ge 3$ as before that also $e_i\cdot e_j\in I_{i+j}$ for all $i,j\ge 3$. Now we can apply 
Lemma $\ref{3.3}$ for $\ell=0$, then for $\ell=1$ and so on. Finally we obtain that $e_i\cdot e_j\in I_{n+1}=0$ for all pairs
$(i,j)\neq (1,1),(1,2),(2,1),(2,2)$. This exactly says that $\Lg\cdot [\Lg,\Lg]=0$, i.e., that $L(e_i)=0$ for all $i\ge 3$.
Hence $(V,\cdot)$ is associative and a Poisson-admissible algebra. \\[0.2cm]
{\em Case 1a2:}  $\Lg^{\lceil \frac{n-4}{2}\rceil}$ is not abelian. The proof is the same as above except for a slight
modification. We have the additional non-trivial Lie brackets $[e_i,e_{n-i+1}]=(-1)^{i+1}\al e_n$ for $2\le i\le n-1$ with
$\al\neq 0$. Therefore we can show $[e_1\cdot e_1,e_i]\in I_{i+3}$ only for all $2\le i\le n-3$, but not for $i=n-2$.
Consequently the first subdiagonal of $L(e_1)$ is not necessarily zero, but of the form $(0,\ldots,0,\la)$ for some
$\la\in \C$. Similarly for $L(e_2)$ it is of the form $(0,\ldots ,0,\mu)$. By induction we see that
\[
e_i\cdot e_j \in 
\begin{cases}
I_{i+j}, & \text{ if } i+j\neq n+1,\\
I_n, & \text{ if } i+j=n+1.
\end{cases}
\]
We want to show that $e_i\cdot e_j\in I_{i+j}$ for $2\le i,j\le n$. We already know this, except for the case where
$i\ge 2,j\ge 3$ with $i+j=n+1$, i.e., with $j=n-i+1$. But then $e_i\cdot e_{j-1}=e_i\cdot e_{n-i}\in I_n$ and 
$e_i\cdot e_j=e_i\cdot e_{n-i+1}\in I_n$, so that we can
write $e_i\cdot e_{j-1}=\zeta e_n$ and $e_i\cdot e_j=\eta e_n$. We have
\begin{align*}
\eta e_n & = e_j\cdot e_i \\
         & = e_1\cdot (e_{j-1}\cdot e_i)-e_{j-1}\cdot (e_1\cdot e_i) \\
         & = \zeta e_1\cdot e_n \\
         & =0,
\end{align*}
so that $\eta=0$ and hence $e_i\cdot e_j=0$ for these $i,j$. Hence we have $e_i\cdot e_j\in I_{i+j}$ for all $2\le i,j\le n$.
Similarly we see that $e_1\cdot e_j\in I_{j+2}$ for all $j\ge 3$. Now we can apply repeatedly Lemma $\ref{3.3}$, starting
with $\ell=0$. It follows again that $\Lg\cdot [\Lg,\Lg]=0$ and we are done. \\[0.2cm]
{\em Case 1b:} It holds $\Lg\cdot \Lg\subseteq I_3$ and $\Lg\cdot I_2\nsubseteq I_4$. It follows that
$\Lg$ is necessarily metabelian, which gives a contradiction. Hence this case cannot occur. The method of proof is
very similar to the one of case $1a$, so that we will only sketch a few steps. For a detailed proof see Proposition
$A.2$ in \cite{END}. First one can write $e_1\cdot e_3=\la e_4+z$ for some $\la\in \C$ and $z\in I_5$. Because of
\[
e_1\cdot e_3=e_1\cdot (e_2\cdot e_1)-e_2\cdot (e_1\cdot e_1)
\] 
we see that $\la e_4+z=\la^2e_4+w$ for $z,w\in I_5$ so that $\la(\la-1)=0$. By assumption $\Lg\cdot I_2\nsubseteq I_4$,
so that $\la\neq 0$ and hence $\la=1$. This implies that both $e_{\frac{n}{2}}\cdot e_{\frac{n}{2}}$, 
$e_{\frac{n}{2}}\cdot e_{\frac{n+2}{2}}$ are multiples of $e_n$ and $[e_{\frac{n}{2}},e_{\frac{n}{2}}]=0$. It follows that
$\Lg^{\lceil \frac{n-4}{2}\rceil}$ is abelian. Using several involved inductions we finally obtain that $\Lg$ is metabelian
and that $[\Lg,\Lg]\cdot [\Lg,\Lg]=0$. \\[0.2cm]
{\em Case 2:} It holds $\Lg\cdot \Lg\nsubseteq I_3$. In this case one can even show that $\Lg$ is isomorphic to the
standard graded filiform Lie algebra $L_n$, defined by the Lie brackets $[e_1,e_i]=e_{i+1}$ for $2\le i\le n-1$, see
Proposition $A.4$ in \cite{END}. Since $L_n$ is metabelian, we obtain a contradiction. This finishes the proof.
\end{proof}

\begin{rem}
By the theorem every CPA-structure on a non-metabelian filiform Lie algebra with adapted basis $(e_1,\ldots ,e_n)$
is of a very simple form. Because of $\Lg\cdot [\Lg,\Lg]=0$ the left multiplications $L(e_i)$ are zero for $i=3,\ldots ,n$,
so that the only non-zero products are given by $e_1\cdot e_1, e_1\cdot e_2=e_2\cdot e_1,e_2\cdot e_2$, which lie in
$Z([\Lg,\Lg])$ by Lemma $\ref{2.7}$. For example, in dimension $6$ we have $Z([\Lg,\Lg])=I_5={\rm span}\{e_5,e_6\}$, 
which recovers the result of our direct computation at the beginning of the proof of the theorem.
\end{rem}

The following example shows that Theorem $\ref{3.4}$ cannot be extended to other nilpotent stem Lie algebras, which 
are not filiform.

\begin{ex}
There exists a nilpotent stem Lie algebra $\Lg$ of solvability class $3$ and nilpotency class $5$ 
admitting a CPA-structure $(V,\cdot)$ which is not associative.
\end{ex}

Let $(e_1,\ldots ,e_9)$ be a basis of $V$ and define the Lie brackets of $\Lg$ by
\begin{align*}
[e_1,e_2] & = e_3, \quad  [e_2,e_3]= e_7, \\
[e_1,e_3] & = e_4, \quad  [e_2,e_5] =- e_9,\\
[e_1,e_4] & = e_5, \quad  [e_3,e_4] = e_9,\\
[e_1,e_6] & = e_8.
\end{align*}
A CPA-structure, which is not associative since $(e_2\cdot e_1)\cdot e_1-e_2\cdot (e_1\cdot e_1)=e_8$, is defined as follows:
\begin{align*}
e_1\cdot e_1 & = e_6, \quad  e_1\cdot e_6= e_8, \\
e_1\cdot e_2 & = e_6, \quad e_2\cdot e_2 =e_6,\\
e_1\cdot e_3 & = e_8, \quad  e_2\cdot e_3 = e_8. \\
\end{align*}
We also see that Theorem $\ref{3.4}$ does not hold for metabelian filiform Lie algebras, because we have the
following result.

\begin{prop}
Let $\Lg$ be a metabelian filiform Lie algebra of dimension $n\ge 4$. Then there exists a CPA-structure
$(V,\cdot)$ on $\Lg$ which is not associative. 
\end{prop}

\begin{proof}
According to \cite{BRA} there exists an adapted basis $(e_1,\ldots , e_n)$ for $\Lg$ such that the Lie brackets
are given by
\begin{align*}
[e_1,e_i] & = e_{i+1}, \; 1\le i\le n-1 \\
[e_2,e_k] & = \al_{2,5}e_{2+k}+\cdots + \al_{2,n-k+3}e_n,\; 3\le k\le n-2, 
\end{align*}
with structure constants $\{\al_{2,k}\mid 5\le k\le n\}$. By convention we set these constants equal to zero for
$n=4$. We define an algebra $(V,\cdot)$ as follows:
\begin{align*}
e_1\cdot e_i & =[e_1,e_i],\; 1\le i\le n,\\
e_2\cdot e_j & =[e_2,e_j],\; 3\le j\le n,\\
e_2\cdot e_2 & = 2\al_{2,5}e_4+\cdots + 2\al_{2,n}e_{n-1}.
\end{align*}
It is easy to verify that $(V,\cdot)$ defines a CPA-structure on $\Lg$. It satisfies $[\Lg,\Lg]\cdot [\Lg,\Lg]=0$, but
not $\Lg\cdot [\Lg,\Lg]=0$.
\end{proof}

Note that the proposition does not hold for $n=3$, because all CPA-structures on the Heisenberg Lie algebra $\Ln_3(\C)$ are
associative, see Proposition $6.3$ in \cite{BU51}.

\begin{rem}\label{3.8}
One can show in general that all CPA-structures on a filiform Lie algebra $\Lg$ satisfy
$[\Lg,\Lg]\cdot [\Lg,\Lg]=0$. If $\Lg$ is not metabelian, this follows from Theorem $\ref{3.4}$. For metabelian
filiform Lie algebras it is proved in \cite{END}.
\end{rem}

We can apply the results now for special classes of filiform Lie algebras and determine all CPA-structures
explicitly. The explicit form is less complicated and is often better suited for applications than a classification 
up to CPA-isomorphism. Therefore we do not give such a classification. \\[0.2cm]
We consider the classes $L_n,Q_n,R_n,W_n$ of filiform Lie algebras discussed in \cite{GOK}, Chapter $4$.
We always assume that $(e_1,\ldots ,e_n)$ is an adapted basis.

\begin{defi}
The Lie algebra $L_n$ for $n\ge 3$ is defined by the Lie brackets
\[
[e_1,e_i]=e_{i+1}, \quad 2\le i\le n-1.
\]
The Lie algebra $Q_n$ for $n\ge 6$ even is defined by the Lie brackets
\begin{align*}
[e_1,e_i] & = e_{i+1}, \quad 2\le i\le n-1,\\
[e_i,e_{n-i+1}] & = (-1)^{i+1}e_n, \quad 2\le i\le \frac{n}{2}.
\end{align*}
The Lie algebra $R_n$ for $n\ge 5$ is defined by the Lie brackets
\begin{align*}
[e_1,e_i] & = e_{i+1}, \quad 2\le i\le n-1,\\
[e_2,e_i] & = e_{i+2}, \quad 3\le i\le n-2.
\end{align*}
The Witt Lie algebra $W_n$  for $n\ge 5$ is defined by the Lie brackets
\begin{align*}
[e_1,e_j] & = e_{j+1}, \quad 2\le j\le n-1,\\[0.1cm]
[e_i,e_j] & = \frac{6(j-i)}{j(j-1)\binom{j+i-2}{i-2}} e_{i+j}, \quad 2\le i\le \frac{n-1}{2},\; i+1\le j\le n-i.
\end{align*}
\end{defi}

To give a CPA-structure $(V,\cdot)$ on $\Lg$ explicitly it is enough to list the non-zero products $e_i\cdot e_j$
for all $1\le i\le j\le n$.

\begin{prop}\label{3.10}
Every CPA-structure on $L_n$, $n\ge 5$ with an adapted basis $(e_1,\ldots ,e_n)$ is either of type $1$ with products
\begin{align*}
e_1\cdot e_1 & = \al_2e_2+\cdots +\al_ne_n,\\
e_1\cdot e_2 & = \be e_{n-1}+\ga e_n,\\
e_1\cdot e_3 & = \be e_n,\\ 
e_2\cdot e_2 & = \de e_n,
\end{align*}
with arbitrary parameters $\al_i,\be,\ga,\de$ satisfying $\al_2\de+\be=0$, or of type $2$ with products
\begin{align*}
e_1\cdot e_1 & = \al_2e_2+\cdots +\al_ne_n,\\
e_1\cdot e_2 & = e_3+ \be e_{n-1}+\ga e_n,\\
e_1\cdot e_3 & = e_4+ \be e_n,\\ 
e_1\cdot e_k & = e_{k+1},\; 4\le k\le n-1,\\
e_2\cdot e_2 & = \de e_n,
\end{align*}
with arbitrary parameters $\al_i,\be,\ga,\de$ satisfying $\al_2\de-\be=0$. In both cases we have 
$[L_n,L_n]\cdot [L_n,L_n]=0$, but not $L_n\cdot [L_n,L_n]=0$ in general.
\end{prop}

\begin{proof}
A direct verification shows that the above products indeed define a CPA-structure on $L_n$ for all
given parameters. Conversely let $(V,\cdot)$ be a CPA-structure on $\Lg$. We will show by induction on $n\ge 5$
that $(V,\cdot)$ is either of type $1$ or of type $2$. For $n=5$ this follows from a direct computation.
For the induction step we use that $L_n/Z(L_n)\cong L_{n-1}$. Since $\dim (Z(L_n))=1$ it follows $L_n\cdot Z(L_n)=0$
from Corollary $3.4$ of \cite{BU57}. Hence $(V,\cdot)$ induces a CPA-structure on $L_{n-1}$, which by
induction hypothesis is already of type $1$ or of type $2$. We may assume that it is of type $1$, because the proof
for type $2$ works exactly the same way. Hence we know that the products $e_i\cdot e_j$ are of the required form up to
a certain multiple of $e_n$. Now since all left multiplications are derivations of $L_n$ and since we have an explicit
basis of $\Der(L_n)$ with respect to $(e_1,\ldots ,e_n)$, it follows that
\[
[L_n,L_n]\cdot [L_n,L_n]=0,\; I_2\cdot I_4=0,\; L_n\cdot I_5=0.
\]
Indeed, by Proposition $1$ in Chapter $4$ of \cite{GOK}, a basis of $\Der(\Lg)$ with respect to
$(e_1,\ldots ,e_n)$ is given by the $2n-1$ endomorphisms 
\[
\ad (e_1),\ldots ,\ad(e_{n-1}), t_1,t_2,t_3,h_2,\cdots ,h_{n-2}
\]
defined by
\begin{alignat*}{2}
t_1(e_1) & = 0,  \quad &  t_1(e_i) & = e_i,\; 2\le i\le n, \\
t_2(e_1) & = e_1, \quad &  t_2(e_i) & = (i-1)e_i,\; 2\le i\le n, \\
t_3(e_1) & = e_2, \quad &  t_3(e_i) & = 0,\; 2\le i\le n, \\
h_k(e_1) & = 0,  \quad  &  h_k(e_i) & = e_{i+k}, \; 2\le i\le n-k, \\
h_k(e_j) & = 0,  \quad  & n-k & < j\le n\\
\end{alignat*}
for all $2\le k\le n-2$. Hence we see that the non-zero products  $e_i\cdot e_j$ are given as follows:
\begin{align*}
e_1\cdot e_1 & = \al_2e_2+\cdots +\al_{n-1}e_{n-1}+\zeta_1e_n,\\
e_1\cdot e_2 & = \be e_{n-2} +\ga e_{n-1}+\zeta_2e_n,\\
e_1\cdot e_3 & = \be e_{n-1} +\zeta_3e_n,\\
e_1\cdot e_4 & = \zeta_4e_n,\\
e_2\cdot e_2 & = \de e_{n-1} +\zeta_5e_n,\\
e_2\cdot e_3 & = \zeta_6e_n.
\end{align*}
Because $L(e_1)$ and $L(e_2)$ are a linear combination of the above derivations, we immediately see that
$\zeta_3=\ga$, $\zeta_4=\beta$ and $\zeta_6=\de$. By \eqref{com5} we have
\begin{align*}
\de e_n & = e_3\cdot e_2 \\
        & = [e_1,e_2]\cdot e_2 \\
        & = e_1\cdot (e_2\cdot e_2)-e_2\cdot (e_1\cdot e_2) \\
        & = e_1\cdot (\de e_{n-1}+\zeta_5 e_n)-e_2\cdot (\be e_{n-2} +\ga e_{n-1}+\zeta_2e_n)
\end{align*}
which is equal to $0$ for $n\ge 5$. Hence we have $\de=0$. Similarly we obtain
\begin{align*}
\be e_{n-1} +\zeta_3e_n & = e_3\cdot e_1 \\
                      & = [e_1,e_2]\cdot e_1 \\
                      & = e_1\cdot (e_2\cdot e_1)-e_2\cdot (e_1\cdot e_1) \\
        & = e_1\cdot (\be e_{n-2} +\ga e_{n-1}+\zeta_2e_n)-e_2 \cdot (\al_2e_2+\cdots +\al_{n-1}e_{n-1}+\zeta_1e_n)\\
        & = -\al_2\zeta_6e_n,
\end{align*}
which gives $\be=0$ and $\al_2\zeta_6+\zeta_3=0$. So we obtain exactly the CPA-structure of type $1$ on $L_n$.
\end{proof}

\begin{rem}\label{3.11}
For $n=4$ these two types of CPA-structures in Proposition $\ref{3.10}$ can be merged, but with a different condition, 
namely $\be(\be-1)=\al_2\de$ and $L(e_4)=0$,
\[
L(e_1)=\begin{pmatrix} 0 & 0 & 0 & 0 \cr \al_2 & 0 & 0 & 0 \cr \al_3 & \be & 0 & 0 \cr \al_4 & \ga & \be & 0 \end{pmatrix},\; 
L(e_2)=\begin{pmatrix} 0 & 0 & 0 & 0 \cr 0 & 0 & 0 & 0 \cr \be & 0 & 0 & 0 \cr \ga & \de & 0 & 0 \end{pmatrix},\; 
L(e_3)=\begin{pmatrix} 0 & 0 & 0 & 0 \cr 0 & 0 & 0 & 0 \cr 0 & 0 & 0 & 0 \cr \be  & 0 & 0 & 0 \end{pmatrix}.
\]
\end{rem}

\begin{prop}\label{3.12}
Every CPA-structure on $Q_n$, $n\ge 6$ even, with an adapted basis $(e_1,\ldots ,e_n)$ is given as follows:
\begin{align*}
e_1\cdot e_1 & = \al e_{n-1}+\be e_n,\\
e_1\cdot e_2 & = -\al e_{n-1}+\ga e_n,\\
e_2\cdot e_2 & = \al e_{n-1}+ \de e_n,
\end{align*}
with arbitrary parameters $\al,\be,\ga,\de$. 
\end{prop}

\begin{proof}
Let $(V,\cdot)$ be a CPA-structure on $Q_n$. Then $Q_n\cdot [Q_n,Q_n]=0$ by Theorem $\ref{3.4}$. Now using  
a basis of $\Der(Q_n)$ we obtain that the products are given as above by a straightforward calculation. 
\end{proof}

In the same way, by using the explicit form ot the derivation algebra, one can show the following result.

\begin{prop}\label{3.13}
Every CPA-structure on $R_n$, $n\ge 6$ with an adapted basis $(e_1,\ldots ,e_n)$ is either of type $1$ with products
\begin{align*}
e_1\cdot e_1 & = \al_3e_3+\cdots +\al_ne_n,\\
e_1\cdot e_2 & = \al_3e_4+\cdots +\al_{n-2}e_{n-1}+\be e_n,\\
e_2\cdot e_2 & = \al_3e_5+\cdots +\al_{n-3}e_{n-1}+\ga e_n, 
\end{align*}
or of type $2$ with products
\begin{align*}
e_1\kringel e_i  & = e_1\cdot e_i+[e_1,e_i],\; 1\le i\le n,\\
e_2\kringel e_2  & = 2e_4+ e_2\cdot e_2,\\
e_2\kringel e_i  & = [e_2,e_i],\; 3\le i\le n,\\ 
\end{align*}
where $e_i\cdot e_j$ is a CPA-structure of type $1$.
\end{prop}

\begin{rem}
There is a third type of CPA-structure on $R_n$ for $n=5$, given by
\begin{align*}
e_1\cdot e_1 & = -\frac{1}{2}e_2+ \al e_3+ \be e_4 + \ga e_5,\\
e_1\cdot e_2 & = \frac{1}{2} e_3+\de e_4 +\ep e_5,\\
e_1\cdot e_3 & = \frac{1}{2} e_4+(\de -\al) e_5,\\
e_2\cdot e_2 & = \frac{1}{2} e_4+(\de -\al) e_5.
\end{align*}
\end{rem}

Finally, the result for $W_n$ is given as follows.

\begin{prop}\label{3.15}
Every CPA-structure on $W_n$, $n\ge 7$ with an adapted basis $(e_1,\ldots ,e_n)$ is given as follows:
\begin{align*}
e_1\cdot e_1 & = \al e_{n-2}+\be e_{n-1}+ \ga e_n,\\
e_1\cdot e_2 & = \frac{6(n-4)}{(n-2)(n-3)} \al e_{n-1}+\de e_n,\\
e_2\cdot e_2 & = \ep e_n.
\end{align*}
\end{prop}

\section{CPA-structures on Lie algebras of strictly upper-triangular matrices}

In this section we study CPA-structures on the Lie algebra $\Ln_n(K)$ of strictly upper-triangular $n\times n$-matrices over a 
field $K$. This Lie algebra is nilpotent of class $c=n-1$ and dimension $\frac{n(n-1)}{2}$. It has a basis 
$\{ E_{j,k} \mid 1\le j<k\le n\}$, where the matrices $E_{i,j}$ have an entry $1$ at position $(i,j)$ and $0$ otherwise. 
The non-trivial Lie brackets are given by
\begin{align*}
[E_{j,k},E_{k,\ell}] & = E_{i,\ell},\; 1\le j<k<\ell\le n.
\end{align*}

The following lemma is well known.

\begin{lem}\label{4.1}
Let $\Lg$ be the Lie algebra $\Ln_n(K)$ with $n\ge 4$ and suppose that $[z,[\Lg,\Lg]]=0$ for some $z\in \Lg$. Then we have
$z\in \Lg^{n-4}$.
\end{lem}

Here is the main result of this section.

\begin{thm}\label{4.2}
Let $\Lg$ be the Lie algebra $\Ln_n(K)$ with $n\ge 5$. Then every CPA-structure $(V,\cdot)$ on $\Lg$ is associative 
and the algebra $(V,\cdot)$ is Poisson-admissible. Moreover we have $\Lg\cdot \Lg\subseteq \Lg^{n-3}$.
\end{thm}

\begin{proof}
Since $\Lg$ is a nilpotent stem Lie algebra, all left multiplications are nilpotent by 
Theorem $3.6$ in \cite{BU57}. Hence they are nilpotent derivations of $\Lg$. By Theorem $3.2$ in \cite{OWY}
there exists for every nilpotent derivation $D \in \Der(\Lg)$ an $u\in \Lg$ and a $\psi\in \Der(\Lg)$ such that
$D=\ad (u)+\psi$ and $\psi(\Lg)\subseteq \Lg^{n-3}$, $\psi([\Lg,\Lg])=0$. Hence for every $x\in \Lg$ there is a $z\in \Lg$
and a $\psi\in \Der(\Lg)$, depending on $x$, such that $L(x)=\ad(z)+\psi$ with these properties. We show by induction
over $n$ that $\Lg\cdot [\Lg,\Lg]=0$. For $n=5,6$ this can be verified by a direct computation. So we may assume
$n\ge 7$ for the induction step. Define Lie ideals $I$ and $J$ in $\Lg$ as follows
\begin{align*}
I  & = {\rm span} \{ E_{1,i} \mid 2\le i\le n\},\\
J  & = {\rm span} \{ E_{i,n} \mid 1\le i\le n-1\}.
\end{align*}
Let $\La=I+\Lg^{n-3}$,  $\Lb=J+\Lg^{n-3}$ and note that $\Lg^{n-3}={\rm span}\{E_{1,n-1},E_{2,n},E_{1,n} \}$.
Then we have $\Lg/\La\cong \Lh/Z(\Lh)$ with $\Lh=\Ln_{n-1}(K)$. By induction hypothesis every CPA-structure on $\Lh$ satisfies
$\Lh\cdot [\Lh,\Lh]=0$ and $\Lh\cdot \Lh\subseteq \Lh^{n-4}$. Let $x\in \Lg$. Then we have
\begin{align*}
x\cdot \La & = L(x)(\La) \\
           & = \ad(z)(\La)+ \psi(\La) \\
           & \subseteq \La + \Lg^{n-3},
\end{align*}
so that $\Lg\cdot \La\subseteq \La$ because of $\Lg^{n-3}\subseteq \La$. Now let $(V,\cdot)$ be a CPA-structure on $\Lg$.
It induces a CPA-structure on the quotient $\Lg/\La\cong \Lh/Z(\Lh)$, so that we have 
$(\Lg/\La)\cdot (\Lg/\La)\subseteq Z(\Lg/\La)$ and $\Lg/\La\cdot [\Lg/\La,\Lg/\La]=0$. This implies that
$\Lg\cdot \Lg \subseteq Z(\Lg/\La)+\La$ and $\Lg\cdot [\Lg,\Lg]\subseteq \La$. On the other hand we also have $\Lg/\Lb\cong \Lh/Z(\Lh)$,
so that $\Lg\cdot \Lg \subseteq Z(\Lg/\Lb)+\Lb$ and $\Lg\cdot [\Lg,\Lg]\subseteq  \Lb$. Together this yields
\begin{align*}
\Lg\cdot \Lg & \subseteq (Z(\Lg/\La)+\La) \cap (Z(\Lg/\Lb)+\Lb) = \Lg^{n-4},\\
\Lg\cdot [\Lg,\Lg] & \subseteq \La\cap \Lb = \Lg^{n-3}.
\end{align*}
Using this we obtain by \eqref{com6} that
\begin{align*}
\Lg\cdot \Lg^2 & \subseteq [\Lg\cdot \Lg,[\Lg,\Lg]]+[\Lg,\Lg\cdot [\Lg,\Lg]] \\
               & \subseteq [\Lg,[\Lg,\Lg\cdot \Lg]]+[\Lg,[\Lg\cdot \Lg,\Lg]]+ [\Lg,\Lg^{n-3}] \\
               & \subseteq \Lg^{n-2}.
\end{align*}
This implies similarly that
\begin{align*}
\Lg\cdot \Lg^3 & \subseteq [\Lg\cdot \Lg,\Lg^2]+[\Lg,\Lg\cdot \Lg^2]  \\
               & \subseteq [\Lg,[[\Lg,\Lg],\Lg\cdot \Lg]+[[\Lg,\Lg],[\Lg\cdot \Lg,\Lg]]+[\Lg,\Lg^{n-2}] \\
               & \subseteq \Lg^{n-1}=0.
\end{align*}
Since $n\ge 7$ we have $\Lg^{n-4}\supseteq \Lg^3$ and hence $\Lg\cdot \Lg^{n-4}=0$. Then, by \eqref{com5},
\[
\Lg\cdot [\Lg,\Lg] \subseteq \Lg\cdot (\Lg\cdot \Lg)\subseteq \Lg\cdot \Lg^{n-4}=0. 
\] 
It remains to show that $\Lg\cdot \Lg\subseteq \Lg^{n-3}$ and not only $\Lg\cdot \Lg\subseteq \Lg^{n-4}$. Suppose
that $L(x)(\Lg)$ is not contained in $\Lg^{n-3}$ for some $x\in \Lg$. Then there is a $z\in \Lg$ and a $\psi\in \Der(\Lg)$
such that $[z,\Lg]+\psi (\Lg)\nsubseteq \Lg^{n-3}$. Then $[z,\Lg]\nsubseteq \Lg^{n-3}$ because of $\psi(\Lg)\subseteq \Lg^{n-3}$,
and hence $z\not\in \Lg^{n-4}$. By Lemma $\ref{4.1}$ it follows that $[z,[\Lg,\Lg]]\neq 0$. Since $\psi([\Lg,\Lg])=0$ this 
implies $L(z)([\Lg,\Lg])\neq 0$, which contradicts $\Lg\cdot [\Lg,\Lg]=0$.
\end{proof}

Note that the result does not hold for $n=4$. There are CPA-structures on $\Ln_4(K)$, which are not associative.

\begin{rem}
One can extend this result to the solvable Lie algebra $\Lt_n(K)$ of all $n\times n$ upper-triangular matrices
over $K$, see Proposition $5.34$ in \cite{END}. Every CPA-structure $(V,\cdot)$ on $\Lt_n(K)$ for $n\ge 3$ is associative and
satisfies $\Lg\cdot \Lg\subseteq Z(\Lg)+Z([\Lg,\Lg])$.
\end{rem}

\section*{Acknowledgments}
Dietrich Burde is supported by the Austrian Science Foun\-da\-tion FWF, grant P28079 and grant I3248 and Christof Ender is 
supported by the Austrian Science Foun\-da\-tion FWF, grant P28079.

\end{document}